\documentclass[a4paper,twoside,11pt, reqno]{amsart} 
\newcommand{\Ueberschrift} {Elementary anabelian varieties are anabelian}
\newcommand{\Kurztitel}{Elementary anabelian varieties are anabelian}
\usepackage[normal]{phaine_style}
\usepackage[margin=1.5in]{geometry}
\usepackage{mathrsfs}
\usepackage{tikz-cd}
\newtheorem{thmABC}{Theorem}
\newtheorem{defABC}[thmABC]{Definition}

\usepackage{amssymb}
\usepackage[headings]{fullpage}
\usepackage{comment}
\usepackage[all]{xy} \CompileMatrices

\usepackage{extarrows}

\DeclareMathOperator{\rk}{rk}

\DeclareMathOperator{\spez}{sp}

\DeclareMathOperator{\out}{out}
\DeclareMathOperator{\HT}{HT}

\newcommand{\cB}{{\mathscr B}}

\DeclareMathOperator{\im}{im}

\DeclareMathOperator{\open}{open}
\DeclareMathOperator{\dom}{dom}
\newcommand{\cohdom}{\mathrm{ci}}
\newcommand{\ecohdom}{\mathrm{eci}}
\newcommand{\secohdom}{\mathrm{seci}}

\newcommand{\scohdom}{\mathrm{sci}}

\newcommand{\HTinj}{\mathrm{HT-inj}}

\newcommand{\RTop}{\categ{RTop}}

\DeclareMathOperator{\PND}{PND}

\begin{document}

\title[\Kurztitel]{\Ueberschrift} 

\author{Magnus Carlson}
\address{Magnus Carlson, Institut f\"ur Mathematik, Goethe--Universit\"at Frankfurt, Robert-Mayer-Stra\ss e~6--8,
60325~Frankfurt am Main, Germany}
\email{\tt carlson@math.uni-frankfurt.de}

	
\date{\normalsize \today}

\begin{abstract} 
We show that isomorphisms of fundamental groups of elementary anabelian varieties -- varieties obtained as iterated fibrations of hyperbolic curves --  over sub-$p$-adic fields correspond bijectively to isomorphisms of varieties. Moreover, dominant maps between proper elementary anabelian varieties are in bijection with ``stably cohomologically injective'' maps of fundamental groups: open maps whose pullbacks to all open subgroups induce injections on cohomology rings with $\ell$-adic coefficients, for any prime $\ell$. This verifies conjectures of Grothendieck from his letter to Faltings. Finally, we establish étale homotopical generalizations of these results.
\end{abstract}

\maketitle
\section{Introduction}
\label{sec:intro}
In his anabelian program \cite{GrothendieckBrief,GrothendieckEsquisse}, Grothendieck conjectured the existence of \emph{anabelian varieties}: varieties completely determined by their étale fundamental group. 
Grothendieck expected hyperbolic curves over a field $K$ finitely generated over $\QQ$ to be anabelian, and formulated the \emph{Hom-conjecture}: given hyperbolic curves $X$ and $Y$ over such a field $K$, the natural map
\[ \Hom^{\dom}_K(X,Y) \rightarrow \Hom_{G_K}^{\open}(\pi_1(X),\pi_1(Y)) \] is a bijection. Here, the left-hand side consists of dominant maps over $K$, while the right-hand side is the set of $\pi_1(Y_{\overline{K}})$-conjugacy classes of open continuous homomorphisms over $G_K$. As a special case, he gave the \emph{Isom-conjecture}: the map \[  \Isom_K(X,Y) \rightarrow \Isom_{G_K}(\pi_1(X),\pi_1(Y)), \]  from the set of $K$-isomorphisms between two hyperbolic curves to the set of $\pi_1(Y_{\overline{K}})$-conjugacy classes of isomorphisms of profinite groups over $G_K$, is also bijective. 

Both of these conjectures are now theorems due to work of Mochizuki \cite{MochizukiProfiniteGrothendieck,MochizukiHom}, following earlier work of Nakamura \cite{NakamuraP1} and Tamagawa \cite{TamagawaGrothendieck},  and are known to hold over any sub-$p$-adic field $K$, i.e., a subfield of a field finitely generated over $\QQ_p$ for some prime number $p$.

It was unclear to Grothendieck which higher-dimensional varieties should be anabelian. However, he expected that \emph{elementary anabelian varieties}, i.e., varieties constructed by successive smooth fibrations from hyperbolic curves, are anabelian. He conjectured analogues of both the Hom-conjecture and the Isom-conjecture in this setting. The main theorem of this paper proves the Isom-conjecture for elementary anabelian varieties. Following terminology suggested by Hoshi \cite{HoshiGrothendieck}, we will primarily refer to such varieties as \emph{hyperbolic polycurves}. 
\begin{thmABC}[see Theorem \ref{cor:isom}] \label{thmintro:isom}
Let $X$ and $Y$ be elementary anabelian varieties (hyperbolic polycurves) over a sub-$p$-adic field $K$. Then the natural map $\Isom_K(X,Y) \rightarrow \Isom_{G_K}(\pi_1(X),\pi_1(Y))$ is bijective.
\end{thmABC} 
We further characterize dominant maps from a smooth proper, geometrically connected $K(\pi,1)$-variety (see Definition \ref{def:kpi1}) into a proper hyperbolic polycurve group-theoretically. This proves a version of the dominant Hom-conjecture\footnote{Grothendieck conjectured that \emph{dominant} maps between hyperbolic polycurves should correspond bijectively to open maps. However, this cannot be true unless $Y$ is a curve; see Example \ref{ex:grothendieckcounter}} for proper hyperbolic polycurves, as conjectured by Grothendieck in his letter to Faltings\footnote{Strictly speaking, Grothendieck conjectured a group-theoretical characterization of dominant maps from an arbitrary smooth variety into a hyperbolic polycurve. In Example \ref{ex:grothendieckcounter}, we prove that no such characterization exists if the hyperbolic polycurve has dimension greater than two. See also the discussion after Theorem \ref{thmintro:dom}}. Our group-theoretical characterization of dominant maps is in terms of \emph{stably cohomologically injective maps} of étale fundamental groups.
\begin{defABC}[see also Definition \ref{def:cohdom}]
Let $K$ be a field of characteristic zero, and let $X$ and $Y$ be geometrically connected varieties over $K$.  A continuous map $\varphi \colon \pi_1(X) \rightarrow \pi_1(Y)$ over $G_K$ is cohomologically injective if it is open and the induced map \[ \varphi^* \colon H^*(\pi_1(Y_{\overline{K}}),\QQ_{\ell}) \rightarrow H^*(\pi_1(X_{\overline{K}}),\QQ_{\ell}) \] is injective for every prime number $\ell$. We say that $\varphi$ is stably cohomologically injective if for every open subgroup $N \subset \pi_1(Y)$, the induced map $\varphi^{-1}(N) \rightarrow N$ is cohomologically injective. 
\end{defABC}
We note that if $\varphi \colon \pi_1(X) \rightarrow \pi_1(Y)$ is a stably cohomologically injective map between étale fundamental groups of proper hyperbolic polycurves, any representative of its $\pi_1(Y_{\overline{K}})$-conjugacy class is stably cohomologically injective. Thus, we say that a $\pi_1(Y_{\overline{K}})$-conjugacy class representing an element of $\Hom_{G_K}(\pi_1(X),\pi_1(Y))$ is stably cohomologically injective if any of its representatives is stably cohomologically injective.
\begin{thmABC}[see Theorem \ref{thm:dominant}] \label{thmintro:dom}
Let $K$ be a sub-$p$-adic field, let $X$ be a smooth, proper, geometrically connected variety which is a $K(\pi,1)$, and let $Y$ be a proper elementary anabelian variety (a proper hyperbolic polycurve) over $K$. Then the natural map
  \[
  \Hom^{\dom}_K(X,Y)
  \longrightarrow
  \ \Hom^{\scohdom}_{G_K}(\pi_1(X),\pi_1(Y))
  \]
  is bijective, where the right-hand side consists of $\pi_1(Y_{\overline{K}})$-conjugacy classes of stably cohomologically injective homomorphisms. If $Y$ has dimension at most two, the same conclusion holds without assuming that $X$ is a $K(\pi,1)$-variety.
\end{thmABC}
As any proper hyperbolic polycurve is a $K(\pi,1)$, this characterizes dominant maps between proper hyperbolic polycurves group-theoretically. Grothendieck made the conjecture that for any smooth variety $X$ and any hyperbolic polycurve $Y$ over a field finitely generated over $\QQ$, the natural map $\Hom^{\dom}_K(X,Y) \rightarrow \Hom_{G_K}^{\open}(\pi_1(X),\pi_1(Y))$ is bijective. When $Y$ is a curve, this follows from the Hom-theorem \cite[Theorem A]{MochizukiHom} of Mochizuki.  

In the case when $\dim Y \geq 2$, any dominant map $f \colon X \rightarrow Y$ over $K$ induces an open map $f_* \colon \pi_1(X) \rightarrow \pi_1(Y)$ over $G_K$, but not conversely. More precisely, in Example \ref{ex:grothendieckcounter}, we show that for any hyperbolic polycurve $Y$ with $\dim Y \geq 2$, there exists a smooth projective variety $X$ and open maps $\pi_1(X) \rightarrow \pi_1(Y)$ that do not come from any dominant map $X \rightarrow Y$. The question arises whether there exists a group-theoretical refinement of the class of open homomorphisms that cuts out exactly the image of $\Hom^{\dom}_K(X,Y)$ inside $\Hom^{\open}_{G_K}(\pi_1(X),\pi_1(Y))$ for every smooth proper $X$ and $Y$ a proper hyperbolic polycurve of arbitrary dimension. Theorem \ref{thmintro:dom} shows that stably cohomologically injective maps provide such a refinement when the dimension of $Y$ is at most two. In Example \ref{ex:grothendieckcounter}, we show that no such refinement can exist unless $Y$ has dimension at most two. Indeed, we show that if $\dim(Y) \geq 3$, then there is a smooth proper variety $X$ together with two morphisms \[ f,g \colon X \rightarrow Y \]  such that the induced maps on fundamental groups are isomorphic over $G_K$, while one of $f$ and $g$ is dominant and the other one is not. We are led to believe that the class of smooth proper $K(\pi,1)$-varieties is close to being the largest class for which there is a group-theoretical characterization of dominant maps.

On the other hand, we obtain a characterization of dominant morphisms from an arbitrary smooth proper variety into a proper hyperbolic polycurve in terms of étale homotopy theory. This provides an affirmative answer to Grothendieck's stronger dominant Hom-conjecture for proper hyperbolic polycurves, if one works with étale homotopy types. 
\begin{thmABC}[see Theorem \ref{thm:dominant}] \label{thmintro:dometale}
Let $K$ be a sub-$p$-adic field, and let $X$ be a smooth, proper, geometrically connected variety over $K$ and $Y$ a proper hyperbolic polycurve over $K$. Denote by $\Pi(X),\Pi(Y)$ their étale homotopy types over $BG_K$, the classifying space of $G_K$. Then the natural map
  \[
  \Hom^{\dom}_K(X,Y)
  \longrightarrow
  \ \pi_0(\Map^{\scohdom}_{BG_K}(\Pi(X),\Pi(Y))),
  \]
from the set of dominant maps to the homotopy classes of stably cohomologically injective maps (see the discussion after Definition \ref{def:cohdom}) $\Pi(X) \rightarrow \Pi(Y)$ over $BG_K$, is bijective.

\end{thmABC}

The Isom-conjecture for hyperbolic polycurves of dimension at most two over sub-$p$-adic fields was proven by Mochizuki in \cite{MochizukiHom}. Hoshi \cite{HoshiGrothendieck} then made a fundamental study of hyperbolic polycurves and proved the Isom-conjecture for hyperbolic polycurves of dimension at most four over sub-$p$-adic fields. For hyperbolic polycurves with the same restriction on dimension, he further proved that certain morphisms of fundamental groups arise from morphisms of varieties. In another direction, Schmidt--Stix \cite{SchmidtStix} proved the Isom-conjecture for \emph{strongly hyperbolic Artin neighborhoods} over fields finitely generated over $\QQ$. Note that strongly hyperbolic Artin neighborhoods constitute a well-behaved class of hyperbolic polycurves. The Isom-conjecture was further proven for certain restricted classes of hyperbolic polycurves in several articles \cite{NagamachiIppei,TakaoAnabelian,SawadaGrothendieck,MochizukiTopicsII}.
\subsection*{Outline}
We now sketch how Theorems \ref{thmintro:isom}, \ref{thmintro:dom} and \ref{thmintro:dometale} are proved. We first show that any hyperbolic polycurve $X$ admits a connected finite étale cover $X' \to X$ such that $X'$ has an snc-compactification whose associated log-canonical bundle has a nonzero global section. To prove this, we use semipositivity results for canonical bundles in conjunction with a Riemann--Roch argument. This is the key geometric input needed for our argument. We then prove, using $p$-adic Hodge theory, that any open homomorphism \[ \varphi \colon \pi_1(X) \rightarrow \pi_1(Y)\] over $G_K$  between hyperbolic polycurves over a $p$-adic local field that is stably cohomologically injective is \emph{nondegenerate}, a notion of Hoshi \cite{HoshiGrothendieck} which we now recall. The homomorphism $\varphi$ is \emph{degenerate} if there exist a nonempty Zariski-open $U$ of $X$ and a smooth, surjective and geometrically connected map $g \colon U \rightarrow Z$, with $Z$ normal and $\dim Z < \dim Y$, such that the composite \[ \pi_1(U) \rightarrow  \pi_1(X) \rightarrow \pi_1(Y)\]  factors through $\pi_1(U) \rightarrow \pi_1(Z)$. We call $\varphi$ \emph{nondegenerate} if it is not degenerate. We then use functoriality of the Hodge--Tate decomposition to prove that such a factorization would force the map on Hodge--Tate weight $n = \dim Y$ subspaces, which is induced by $H^n(Y_{\overline{K}},\QQ_p) \otimes \CC_p \rightarrow H^n(U_{\overline{K}},\QQ_p) \otimes \CC_p$, to be zero. After replacing $Y$ with an étale cover, we can assume that $Y$ has a nonzero log-canonical form. Since $\varphi$ is stably cohomologically injective, the map \[ H^n(Y_{\overline{K}},\QQ_p) \otimes \CC_p \rightarrow H^n(U_{\overline{K}},\QQ_p) \otimes \CC_p \] can be shown to be injective on Hodge--Tate weight $n$ subspaces. However, this contradicts the existence of a nonzero log-canonical form on $Y$, since the map is zero on the Hodge--Tate weight $n$ subspaces. We then proceed by proving that stably cohomologically injective morphisms are nondegenerate over any sub-$p$-adic field by spreading out and specializing. The main theorems are then obtained by using this result in conjunction with a result of Hoshi, which characterizes dominant morphisms into a hyperbolic polycurve in terms of nondegenerate morphisms of fundamental groups. 
\subsection*{Notation and conventions}
A \emph{variety} over a field $K$ is a separated, geometrically connected and reduced scheme of finite type over $K$. The étale fundamental group of a variety $X$ is denoted $\pi_1(X)$, and all maps between profinite groups are continuous. We consistently omit basepoints in our notation for étale fundamental groups, since our statements, as well as proofs, are insensitive to the choice of basepoints. For a field $K$, we write $G_K = \Gal(\overline{K}/K)$ for the absolute Galois group of $K$, where $\overline{K}$ is a separable closure of $K$.

\subsection*{Acknowledgments}
I gratefully acknowledge support from the 
 Deutsche Forschungsgemeinschaft (DFG) through the Collaborative Research Centre
  TRR 326 ``Geometry and Arithmetic of Uniformized Structures'', project number  444845124.  I am very thankful to Benjamin Steklov and Ruth Wild for insightful discussions related to this work. I am particularly grateful to Jakob Stix for his careful reading, for his encouragement, and for his many valuable suggestions.

\section{Preliminaries on hyperbolic polycurves}  
\label{sec:prelpoly}
We recall background on hyperbolic polycurves and their fundamental groups, and prove that any hyperbolic polycurve $X$ admits a connected étale cover $X' \rightarrow X$ such that $X'$ has an snc-compactification whose associated log-canonical bundle admits a nonzero global section. The results on fundamental groups of hyperbolic polycurves are proved in Hoshi's paper \cite{HoshiGrothendieck}.
\subsection{Hyperbolic polycurves and their fundamental groups}
\begin{definition}
  Let $S$ be a scheme. A scheme $X$ over $S$ is a \emph{hyperbolic curve over $S$} if there exist:      
  \begin{enumerate}
  \item a smooth, proper, geometrically connected scheme $\overline{X}$ of relative dimension $1$ over $S$, with geometric fibers of genus $g$, and 
  \item a relative divisor $D \subset \overline{X}$ which is finite étale over $S$ of degree $r$,
  \end{enumerate}
  such that $X \cong \overline{X} \setminus D$ over $S$, and the inequality $2-2g-r < 0$ holds. 
  \end{definition}
In the above definition, $D$ is allowed to be empty. Then $X$ is proper over $S$ and we call $X$ a proper hyperbolic curve over $S$. As noted by Mochizuki \cite[Section 0]{MochizukiAbsoluteHyperbolic}, if $S$ is normal, then the compactification $\overline{X}$ of a hyperbolic curve $X$ over $S$ is unique up to canonical isomorphism. We will call the integer $g$ above the \emph{genus} of $X$ over $S$.
\begin{definition} \label{def:polycurves}
  Let $S$ be a scheme. A morphism $p \colon X \rightarrow S$ is a hyperbolic polycurve over $S$ if there exists a factorization \[ X = X_n \xrightarrow{p_n} X_{n-1} \xrightarrow{p_{n-1}} \cdots \xrightarrow{p_2} X_1 \xrightarrow{p_1} X_0 = S \] of $p$ such that each $X_i \xrightarrow{p_i} X_{i-1}$ is a hyperbolic curve over $X_{i-1}$. We call such a factorization a sequence of \emph{parameterizing morphisms}.
\end{definition}
\begin{remark} \label{rmk:etalecoverpoly}
As explained in \cite[Proposition 2.3]{HoshiGrothendieck}, if $S$ is a connected, separated, normal and noetherian scheme over $\Spec \QQ$, then any connected finite étale cover $p \colon Y \rightarrow X$ of a hyperbolic polycurve $X$ over $S$ is a hyperbolic polycurve over some finite étale cover $S' \xrightarrow{p_0} S$. By induction, it is enough to prove this when $X$ is a hyperbolic curve over $S$, and then one defines $S'$ as the normalization of $Y$ in $S$. We see that if \[ X = X_n \rightarrow X_{n-1} \rightarrow \cdots \rightarrow X_1 \rightarrow S\]  is a sequence of parameterizing morphisms of $X \rightarrow S$, then there are parameterizing morphisms \[ Y = Y_n \rightarrow Y_{n-1} \rightarrow \cdots \rightarrow Y_1 \rightarrow S' \] of $Y \rightarrow S'$ so that one has a natural commutative diagram 
\[ \begin{tikzcd} Y = Y_n \arrow[r] \arrow[d,"p=p_n"] & Y_{n-1} \arrow[r] \arrow[d,"p_{n-1}"] & \cdots \arrow[r]  & Y_1 \arrow[r] \arrow[d,"p_1"]   & S' \arrow[d,"p_0"] \\
X = X_n \arrow[r] & X_{n-1} \arrow[r] & \cdots \arrow[r] & X_1 \arrow[r] & S \end{tikzcd} \label{eq:diagramcomm1}  \]  where all vertical maps are finite étale.
\end{remark}
We recall properties of fundamental groups of hyperbolic polycurves over a characteristic zero base. When working with étale fundamental groups, we will omit basepoints in the notation. Since our schemes are connected, the groups are well defined up to inner automorphism, and the statements we consider are insensitive to this ambiguity. A profinite group $G$ is \emph{slim} if every open subgroup is center-free. It is \emph{elastic} if any non-trivial closed topologically finitely generated subgroup of $G$ that is normal in an open subgroup $H$ of $G$ is open in $G$. As noted in \cite[Remark 0.1.3]{MochizukiAbsoluteHyperbolic}, a profinite group $G$ is slim if and only if, for each open subgroup $H$ of $G$, the centralizer $Z_G(H)$ of $H$ in $G$ is trivial. The following result is \cite[Proposition 2.4]{HoshiGrothendieck}. It implies that if $X$ is a hyperbolic polycurve over a field $K$ of characteristic zero, then $\pi_1(X_{\overline{K}})$ is slim (see also the proof of \cite[Corollary 4.10]{NakamuraGaloisRigidity}).
\begin{lemma} \label{lem:pi1ofpoly}
Let $X$ be a hyperbolic polycurve of relative dimension $n$ over a connected normal noetherian scheme $S$ over $\QQ$, and let
\[ X = X_n \xrightarrow{p_n} X_{n-1} \xrightarrow{p_{n-1}} \cdots \xrightarrow{p_2}  X_1 \xrightarrow{p_1} X_0 = S \] be a sequence of parameterizing morphisms. Let $i,j$ be integers such that 
$0 \leq i < j \leq n$. Fix a basepoint $\overline{x}_i$ of $X_i$. Then:
\begin{enumerate}
\item The map $X_j \rightarrow X_i$ gives rise to a short exact sequence 
\[1 \rightarrow \pi_1(X_j \times_{X_i} \overline{x}_i) \rightarrow \pi_1(X_j) \rightarrow \pi_1(X_i) \rightarrow 1 \] of profinite groups.
\item $\pi_1(X_j \times_{X_i} \overline{x}_i)$ is a non-trivial topologically finitely generated profinite group, which is moreover \emph{slim} and torsion-free. 
\item For any basepoint $\overline{x}_{j-1}$ of $X_{j-1}$, the fundamental group of the geometric fiber $\pi_1(X_j \times_{X_{j-1}} \overline{x}_{j-1})$ is elastic. 
\end{enumerate}
\end{lemma}
For two varieties $X,Y$ over a field $K$, we write $\Hom_{G_K}^{\open}(\pi_1(X),\pi_1(Y))$ for the set of $\pi_1(Y_{\overline{K}})$-conjugacy classes of open maps over $G_K$. We denote by $\Hom_{G_K}^{\out,\open}(\pi_1(X_{\overline{K}}),\pi_1(Y_{\overline{K}}))$ the set of conjugacy classes of open maps $\pi_1(X_{\overline{K}}) \rightarrow \pi_1(Y_{\overline{K}})$ which commute with the outer action of $G_K$. We state the following folklore result, for which we could find no precise reference, but see the closely related \cite[2. Corollary]{FaltingsCurves} and \cite[Lemma 7.1]{TamagawaGrothendieck}.
\begin{proposition} \label{prop:homouter}
Let $X$ and $Y$ be varieties over a field $K$. Suppose that $\pi_1(Y_{\overline{K}})$ is slim. Then there is a natural bijection 
\[ \Hom_{G_K}^{\open}(\pi_1(X),\pi_1(Y)) \rightarrow \Hom_{G_K}^{\out,\open}(\pi_1(X_{\overline{K}}),\pi_1(Y_{\overline{K}})), \]  
which sends $\varphi \colon \pi_1(X) \rightarrow \pi_1(Y)$ to the induced map $\pi_1(X_{\overline{K}}) \rightarrow \pi_1(Y_{\overline{K}})$. \end{proposition}
\begin{proof}
We first prove injectivity. Let $\varphi,\psi \colon \pi_1(X) \rightarrow \pi_1(Y)$ be two maps over $G_K$ which define the same element of $\Hom_{G_K}^{\out,\open}(\pi_1(X_{\overline{K}}),\pi_1(Y_{\overline{K}}))$. After replacing $\psi$ with a conjugate, we can assume that $\varphi$ and $\psi$ coincide on $\pi_1(X_{\overline{K}})$. We define the $1$-cocycle  $\delta \colon \pi_1(X) \rightarrow \pi_1(Y_{\overline{K}})$ by the formula $\delta_s = \varphi(s)\psi(s)^{-1}$ for $s \in \pi_1(X)$; we claim that this factors over $G_K$ as a cocycle, i.e., is constant on right and left cosets of $\pi_1(X_{\overline{K}})$. Since $\delta$ is a $1$-cocycle, we obtain that \[\delta_{st} = \delta_s \psi(s) \delta_t \psi(s)^{-1} \] for any $s,t \in \pi_1(X)$.  This formula clearly implies constancy on right cosets, and constancy on left cosets is derived from constancy on right cosets, via the identity $\delta_{st} = \delta_{t (t^{-1}st)} = \delta_t$ for $s \in \pi_1(X_{\overline{K}}), t \in \pi_1(X)$.  If $s \in \pi_1(X_{\overline{K}})$ and $t$ is arbitrary, we derive, since $\delta$ factors through $G_K$, the equality \[ \delta_t = \delta_{st} = \delta_s \psi(s) \delta_t \psi(s)^{-1}  = \psi(s) \delta_t \psi(s)^{-1}. \]  In other words, \[ \delta_t \in Z_{\pi_1(Y_{\overline{K}})}(\im \psi_{| \pi_1(X_{\overline{K}})}) \] for any $t \in \pi_1(X)$, and this centralizer is trivial by assumption, thus $\varphi = \psi$.

We now prove surjectivity. Given $\overline{\varphi} \in  \Hom_{G_K}^{\out,\open}(\pi_1(X_{\overline{K}}),\pi_1(Y_{\overline{K}}))$, we define the graph of $\overline{\varphi}$ as
 \[ \Gamma_{\overline{\varphi}} = \{ (s,t) \in \pi_1(X) \underset{G_K}{\times} \pi_1(Y) | c_t \circ \overline{\varphi} = \overline{\varphi} \circ c_s \}.  \] We note that the projection $\pr_X \colon \Gamma_{\overline{\varphi}} \rightarrow \pi_1(X)$ is an isomorphism. Indeed, it is surjective since $\overline{\varphi}$ is $G_K$-equivariant, and injective since $\overline{\varphi}$ is open and $\pi_1(Y_{\overline{K}})$ is slim. We then define $\varphi \colon \pi_1(X) \rightarrow \pi_1(Y)$ by the formula $\pr_Y \circ \pr_X^{-1}$, where $\pr_Y \colon \Gamma_{\overline{\varphi}} \rightarrow \pi_1(Y)$ is the projection. Clearly, we have the equality $c_{\overline{\varphi}(s)} \circ \overline{\varphi} = \overline{\varphi} \circ c_s$ for $s \in \pi_1(X_{\overline{K}})$, thus  $\varphi$ restricts to $\overline{\varphi}$ on $\pi_1(X_{\overline{K}})$. \end{proof}
We review the definition of a $K(\pi,1)$-scheme (see \cite[Appendix A]{StixProjective} for further details). Let $X$ be a connected scheme, and let $X_{\et}$ and $X_{\fet}$ denote the étale and finite étale sites of $X$, respectively. For any sheaf $\mathcal{F}$ on $X_{\fet}$, the map of sites $k \colon X_{\et} \rightarrow X_{\fet}$ induces comparison maps $\mathcal{F} \rightarrow \mathrm{R}k_* k^* \mathcal{F}$ in the derived category $D^+(X_{\fet})$, the bounded below category of abelian sheaves on $X_{\fet}$.
\begin{definition} \label{def:kpi1}
Let $X$ be a connected noetherian scheme. Then $X$ is called a $K(\pi,1)$-scheme if for every torsion sheaf $\mathcal{F}$ on $X_{\fet}$, the comparison map $\mathcal{F} \to \mathrm{R}k_* k^* \mathcal{F}$ is an isomorphism. 
\end{definition}
It is well known that, after choosing a basepoint $\overline{x}$ of $X$, the site $X_{\fet}$ identifies with the category of finite sets equipped with a continuous $\pi_1(X,\overline{x})$-action. One then proves that $X$ is a $K(\pi,1)$ if and only if, for every locally constant constructible sheaf $\mathcal{F}$ on $X$, the comparison map $H^i(\pi_1(X,\overline{x}),\mathcal{F}_{\overline{x}}) \rightarrow H^i(X,\mathcal{F})$ is an isomorphism for all $i \geq 0$. 

We now recall a modern construction of the étale homotopy type of a scheme $X$ via shape theory, using $\infty$-categories. We then study the relation between the étale homotopy type and the étale fundamental group, as well as the behaviour of the étale homotopy type under specialization. The étale homotopy type was originally introduced by Artin--Mazur \cite{ArtinMazur} as a pro-object in the homotopy category of simplicial sets, and was later refined by Friedlander \cite{FriedlanderEtaleHomotopy}, who constructed the étale topological type: a functorial lift of the étale homotopy type to the category of pro-simplicial sets. The construction we present here is closely related to Friedlander's construction; see, for example, \cite[Corollary 5.6]{HoyoisHigher}. We refer to \cite[Sections 2 and 3]{HolzschuhReal} or \cite{SchlankEtaleHomotopy} for further details on the construction of the étale homotopy type. 

Following terminology introduced by Clausen--Scholze, we denote by $\Spc$ the $\infty$-category of \emph{anima} (i.e. the $\infty$-category of spaces, in the sense of Lurie). 
Denote by $\RTop_{\infty}$ the $\infty$-category of $\infty$-topoi, where morphisms are given by the right adjoints of geometric morphisms. Given an $\infty$-topos $\mathcal{X}$, one can define  its shape $\Pi(\mathcal{X})$. This is an object of $\Pro(\Spc)$, the pro-category of anima. One can identify $\Pro(\Spc)$ with $\Fun^{\fin}(\Spc,\Spc)^{\op}$, the opposite of the category of finite-limit preserving functors $\Spc \rightarrow \Spc$. If $\mathcal{X}$ is an $\infty$-topos and $p = (p^*,p_*) \colon \mathcal{X} \rightarrow \Spc$ is its unique geometric morphism to $\Spc$, then the shape $\Pi(\mathcal{X})$ of $\mathcal{X}$ is  the left-exact functor $p_* \circ p^* \colon \Spc \rightarrow \Spc$, see \cite[Definition 7.1.6.3]{HTT} for more details. The shape assembles into a functor \[ \Pi \colon \RTop_{\infty} \rightarrow \Pro(\Spc). \]  An anima $X$ is $\pi$-finite if $\pi_0(X)$ is finite, and for all points $x \in X$, the homotopy groups $\pi_n(X,x)$ are finite for all $n \geq 1$, and vanish for sufficiently large $n$. The inclusion $\Pro\left(\Spc_{\pi}\right) \rightarrow \Pro(\Spc)$ has a left adjoint, known as profinite completion.
\begin{definition}
Let $X$ be a scheme, and denote by $\Sh_{\infty}(X_{\et})$ its $\infty$-topos of étale sheaves of anima. We define the étale homotopy type $\Pi(X)$ of $X$ to be the profinite completion of the shape of $\Sh_{\infty}(X_{\et})$. 
\end{definition}	
According to our conventions, the étale homotopy type of a scheme $X$ is an object of the $\infty$-category $\Pro\left(\Spc_{\pi}\right)$, the pro-category of $\pi$-finite anima. As in \cite[Definition E.5.2.1]{SAG}, the functors $\pi_n$ of taking $n$th homotopy groups of anima are extended to $\Pro\left(\Spc_{\pi}\right)$, and take values in profinite sets for $n=0$, and profinite groups for $n \geq 1$.

The étale homotopy type will be of relevance since it enables us to prove stronger anabelian theorems for polycurves. A geometric basepoint $\overline{x}$ of the scheme $X$ induces a basepoint of $\Pi(X)$, and then there is a canonical isomorphism $\pi_1(\Pi(X),\overline{x}) \cong \pi_1(X,\overline{x})$. Further, if $X$ is locally Noetherian, then for any constant, finite étale sheaf $A$ of abelian groups, there is a canonical isomorphism $H^i(\Pi(X),A) \cong H^i(X,A)$. In addition, if $X$ is a qcqs scheme over a field $K$ with separable closure $\overline{K}$, there is a fiber sequence $\Pi(X_{\overline{K}}) \rightarrow \Pi(X) \rightarrow BG_K$ in $\Pro\left(\Spc_{\pi}\right)$ \cite[Example 0.7]{HaineFundamentalFiber}. Thus, given qcqs-schemes $X,Y$ over $K$, a map $\varphi \colon \Pi(X) \rightarrow \Pi(Y)$ over $BG_K$ naturally induces a map $\Pi(X_{\overline{K}}) \rightarrow \Pi(Y_{\overline{K}})$, which is further equivariant with respect to the natural $G_K$ action (see \cite[E.6.4.4]{SAG} for details). Finally, if $X$ is a scheme which is connected and noetherian, then $X$ is a $K(\pi,1)$ if and only if the higher homotopy groups $\pi_i(X,\overline{x}):=\pi_i(\Pi(X),\overline{x})$ vanish for all $i > 1$ \cite[Proposition 4.1.4]{AchingerWild}. Thus, $X$ is a $K(\pi,1)$ if and only if the natural map $\Pi(X) \rightarrow B\pi_1(X)$ to the classifying anima of $\pi_1(X)$ is an equivalence. We now note how covering space theory can be formulated via the étale homotopy type.
\begin{construction} \label{construction:coveringspacse}
Let $X$ be a noetherian, connected and separated scheme, and let $\Pi(X)$ be its étale homotopy type. Then the full sub-$\infty$-category $\Pi(X)_{\fet}$ of $\Pro(\Spc_{\pi})_{/\Pi(X)}$ which has as objects $0$-truncated morphisms (\cite[Definition E.4.1.1]{SAG})  with homotopy fibers \emph{finite sets}, is equivalent to a $1$-category. Indeed, that it is a $1$-category follows from \cite[Proposition E.4.6.1]{SAG}. Any finite étale cover $X' \rightarrow X$ induces by applying $\Pi$ a $0$-truncated morphism $\Pi(X') \rightarrow \Pi(X)$ \cite[Lemma 2.1]{SchmidtStix}. Then, the natural functor $\Pi \colon X_{\fet} \rightarrow \Pi(X)_{\fet}$ is an equivalence of categories, as an argument involving the long exact sequence of fundamental groups and covering space theory shows. 

We now note that given two connected qcqs schemes $X$ and $Y$, and a map $\varphi \colon \Pi(X) \rightarrow \Pi(Y)$ of étale homotopy types, there is an induced functor $\varphi^* \colon \Pi(Y)_{\fet} \rightarrow \Pi(X)_{\fet}$, which, via the equivalences $X_{\fet} \cong \Pi(X)_{\fet}$ and $Y_{\fet} \cong \Pi(Y)_{\fet}$, induces a functor $\varphi^* \colon Y_{\fet} \rightarrow X_{\fet}$. We claim that this is an exact functor of Galois categories. The fact that $\varphi^*$ preserves limits is clear, and the fact that finite colimits are preserved follows from \cite[Theorem E.6.3.1]{SAG}. Thus, we obtain from $\varphi$, naturally a  $\pi_1(Y)$-conjugacy class of a map $\pi_1(X) \rightarrow \pi_1(Y)$. More importantly, $\varphi^*$ induces for each $Y' \in Y_{\fet}$, a pullback diagram \begin{equation}\begin{tikzcd} \Pi(X') \arrow[r,"\varphi"] \arrow[d] & \Pi(Y') \arrow[d] \\ \Pi(X) \arrow[r,"\varphi"] & \Pi(Y) \label{eq:diagramcovering} \end{tikzcd} \end{equation} in $\Pro(\Spc_{\pi})$, where $X' = \varphi^*(Y')$.  Any representative map $\varphi \colon \pi_1(X) \rightarrow \pi_1(Y)$ in the conjugacy class, provides, for $Y' $ corresponding to $H \subset \pi_1(Y)$, an analogous diagram as above, where $X'$ corresponds to $\varphi^{-1}(H)$.
\end{construction}

 The following technical lemma, which studies how the étale homotopy type varies in well-behaved families, might be of independent interest. It is well-known in the proper case, see \cite[Corollary 12.13]{ArtinMazur}, \cite[Proposition 2.49]{HaineNonabelian}. For a prime number $p$, we let $(-)^{p'} \colon \Pro\left(\Spc_{\pi}\right) \rightarrow \Pro\left(\Spc_{\pi}\right)$ be the prime-to-$p$-completion (see \cite[Recollection 1.23]{HaineNonabelian} for details). If $p=0$, we define \[ (-)^{p'} \colon \Pro\left(\Spc_{\pi}\right) \rightarrow \Pro\left(\Spc_{\pi}\right) \]  to be the identity functor.
\begin{lemma} \label{lemma:kpi1family}
Let $S$ be the spectrum of a strictly henselian discrete valuation ring with generic point $\eta$ and closed point $s$, and let $\overline{\eta} \rightarrow S$ and $\overline{s} \rightarrow S$ be geometric points lying over $\eta$ and $s$.  Let $p \geq 0$ denote the characteristic of the residue field at $s$. Let $f \colon X \rightarrow S$ be a morphism which factors as $X \xrightarrow{j} \overline{X} \xrightarrow{\overline{f}} S$, where $j$ is an open immersion such that the complement is a relative snc divisor, and where $\overline{f}$ is proper and smooth. Then:
\begin{enumerate}
\item the natural map $\Pi(X_{\overline{\eta}})^{p'} \rightarrow \Pi(X)^{p'}$ in $\Pro\left(\Spc_{\pi}\right)$ is an equivalence.
\item the natural map $\Pi(X_{\overline{s}})^{p'} \rightarrow \Pi(X)^{p'}$ in $\Pro\left(\Spc_{\pi}\right)$ is an equivalence.
\end{enumerate}
Thus, there is an associated specialization equivalence $\spez \colon \Pi(X_{\overline{\eta}})^{p'} \rightarrow \Pi(X_{\overline{s}})^{p'}$ in $\Pro\left(\Spc_{\pi}\right)$. 
\end{lemma}
\begin{proof}
Note that by \cite[Theorem E.3.1.6]{SAG} and Whitehead's theorem \cite[Theorem 4.3]{ArtinMazur}, to prove that the maps are equivalences, it suffices to verify the following two claims.  First, the maps induce a bijection on connected components, and for any choice of basepoints, the induced maps on prime-to-$p$ fundamental groups $\pi^{p'}_1$ are isomorphisms. Second, for any locally constant constructible sheaf $\mathcal{F}$ of order prime-to-$p$ such that the image of the monodromy has prime-to-$p$ order, the induced maps on cohomology groups are isomorphisms.

The first claim follows from \cite[Exposé XIII, Corollaire 2.8-2.9, \S 2.10]{SGA1new}: Corollaire 2.8-2.9 imply that the induced maps on $\pi_0$ are isomorphisms, and that for any basepoints $a \in X_{\overline{\eta}}$ and $b \in X_{\overline{s}}$, the natural maps $\pi_1(X_{\overline{\eta}},a)^{p'} \rightarrow \pi_1(X,a)^{p'}$ and $\pi_1(X_{\overline{s}},b)^{p'} \rightarrow \pi_1(X,b)^{p'}$ are isomorphisms.

Now, let $\mathcal{F}$ be a locally constant constructible sheaf on $X$ of order prime to $p$, and such that the order of the monodromy action is prime to $p$. By \cite[Th. Finitude, Appendice,1.3.1--1.3.3]{SGA41/2}, $f \colon X \rightarrow S$ is cohomologically proper for  $\mathcal{F}$, and $R^qf_* \mathcal{F}$ is constructible for any $q \geq 0$. Moreover, since $f$ is smooth, $f$ is locally acyclic for $\mathcal{F}$. The proof of \cite[Exposé XVI, Corollaire 2.3]{SGA4to3} proves that the specialization map associated to $R^q f_* \mathcal{F}$  is an isomorphism, thus $R^q f_* \mathcal{F}$ is locally constant. This implies that all maps in the zig-zag \[ H^*(X_{\overline{\eta}},\mathcal{F}) \rightarrow H^*(X,\mathcal{F}) \leftarrow H^*(X_{\overline{s}},\mathcal{F}) \] are isomorphisms. 
\end{proof}
With notation and assumptions as in Lemma \ref{lemma:kpi1family}, assume that $X \rightarrow S$ has geometrically connected fibers, and that $S$ is a scheme over $\Spec \QQ$. We then have a diagram in $\Pro\left(\Spc_{\pi}\right)$
\[ \begin{tikzcd}  \Pi(X_{\overline{\eta}}) \arrow[r] \arrow[d] & \Pi(X)  \arrow[d] & \arrow[l]  \Pi(X_{\overline{s}}) \arrow[d] \\ 
			B\pi_1(X_{\overline{\eta}}) \arrow[r] & B \pi_1(X) & \arrow[l] B\pi_1(X_{\overline{s}}), \end{tikzcd} \] 
where all squares are commutative, and all horizontal arrows are equivalences. By inverting the rightmost horizontal arrows, we obtain a commutative diagram \[ \begin{tikzcd}  \Pi(X_{\overline{\eta}}) \arrow[r,"\spez"] \arrow[d] & \Pi(X_{\overline{s}}) \arrow[d] \\ 
			B\pi_1(X_{\overline{\eta}}) \arrow[r,"\spez"] & B\pi_1(X_{\overline{s}}) \end{tikzcd} \] in $\Pro\left(\Spc_{\pi}\right)$, where all horizontal arrows are equivalences. We have thus related the specialization maps from Lemma \ref{lemma:kpi1family} to the classical specialization maps of fundamental groups. The above diagram implies that $X_{\overline{\eta}}$ is a $K(\pi,1)$ if and only if $X_{\overline{s}}$ is a $K(\pi,1)$. Further, for any locally constant constructible sheaf $\mathcal{F}$ on $X$, there is for every $n \geq 0$ a commutative diagram \begin{equation} \begin{tikzcd} H^n(X_{\overline{\eta}},\mathcal{F}_{\overline{\eta}}) & \arrow[l,"\spez^*"] H^n(X_{\overline{s}},\mathcal{F}_{\overline{s}}) \\ 
			H^n(\pi_1(X_{\overline{\eta}}),\mathcal{F}_{\overline{\eta}}) \arrow[u] & \arrow[l,"\spez^*"] H^n(\pi_1(X_{\overline{s}}),\mathcal{F}_{\overline{s}}) \arrow[u] , \label{eq:diagram} \end{tikzcd}  \end{equation}
			where the horizontal maps are isomorphisms.
			
The following well-known lemma will be useful to us to describe the set of homotopy classes of maps between étale homotopy types. It follows immediately from  \cite[Proposition 3.19]{HolzschuhReal}. Recall that to any profinite group $G$, one can assign its classifying anima $BG \in \Pro\left(\Spc_{\pi}\right)$.
\begin{lemma} \label{lemma:mappingspaces}
  Let $H$ and $N$ be profinite groups, and let $X \in \Pro\left(\Spc_{\pi}\right)$ be connected. Suppose given maps 
  \[
p \colon X \to BN, \qquad q \colon BH \to BN,
\]
in $\Pro\left(\Spc_{\pi}\right)$, and assume that $\pi_1(q)$ is surjective. Then
  \[
  \pi_0( \Map_{BN}(X, BH))
  \]
  is identified, via taking $\pi_1$, with \[ \Hom_N(\pi_1(X),H) / \ker(q),\]
  with $\ker(q)$ acting by conjugation.
  \end{lemma}

 We conclude this subsection by recalling the following well-known proposition, which is proved exactly as in \cite[Proposition 2.8]{SchmidtStix}.
\begin{proposition} \label{prop:kpi1}
  Let $S$ be a smooth variety over a field of characteristic zero, and let $X$ be a hyperbolic curve over $S$. Then if $S$ is a $K(\pi,1)$, so is $X$.
\end{proposition}
Proposition \ref{prop:kpi1} implies that any hyperbolic polycurve over a field of characteristic zero is a $K(\pi,1)$.
\subsection{Canonical bundles of hyperbolic polycurves}\par
Let $X$ be a hyperbolic polycurve over a field $K$ of characteristic zero. In this subsection, we prove that there is a finite étale cover $Y \rightarrow X$ with the property that $Y$ admits a nonzero log-canonical form. 

Recall that a polycurve $X$ over $K$ has a sequence of parameterizing morphisms \begin{equation} X = X_n \to X_{n-1} \to \cdots \to X_1 \to X_0 = \Spec K. \label{eq:polypara} \end{equation} By \cite[Theorem 1]{SaitoLogSmooth}, this sequence admits a log-smooth compactification
    \begin{equation} (\mathscr{X},X) = (\mathscr{X}_n,X_n) \to (\mathscr{X}_{n-1},X_{n-1}) \to \cdots \to (\mathscr{X}_1, X_1) \to \Spec K, \label{eq:polyparaco} \end{equation} such that $\mathscr{X}_i$ is a regular scheme, proper over $\mathscr{X}_{i-1}$, and the complement $\mathcal{D}_i = \mathscr{X}_i \setminus X_i$ is a normal crossings divisor. Further, $(\mathscr{X}_i,X_i)$ is a log-scheme, with log-structure defined by the divisor $\mathcal{D}_i$, and the map $(\mathscr{X}_i,X_i) \to (\mathscr{X}_{i-1},X_{i-1})$ is log-smooth and extends the family of curves $\overline{X}_i \to X_{i-1}$. We call such a log-smooth compactification of the sequence  \[X = X_n \to X_{n-1} \to \cdots \to X_1 \to X_0 = \Spec K\] a \emph{compactification} of the sequence of parameterizing morphisms, or a compactified sequence of parameterizing morphisms. 
    
Suppose now given a polycurve $X \rightarrow \Spec K$ of dimension $n \geq 1$, and a sequence of parameterizing morphisms as in Diagram \ref{eq:polypara}, and a corresponding compactification as in Diagram \ref{eq:polyparaco}. Let $\overline{\eta} \rightarrow X_1$ be a geometric point lying over the generic point of $X_1$, and for a scheme $Y$ over $X_1$ or $\mathscr{X}_1$, we denote by $Y'$ the base-change to $\overline{\eta}$, for ease of notation. Note that we, by base-changing Diagram \ref{eq:polyparaco} along $\overline{\eta}$, obtain a compactified sequence of parameterizing morphisms for the polycurve $X_{\overline{\eta}} \rightarrow \overline{\eta}$, which is of dimension $n-1$. More precisely, the sequence \[ X'_n \rightarrow X'_{n-1} \rightarrow \cdots \rightarrow X'_2 \rightarrow \overline{\eta} \] is a sequence of parameterizing morphisms for the hyperbolic polycurve $X'_n  \rightarrow \overline{\eta}$ of dimension $n-1$, and \[ (\mathscr{X}',X') = (\mathscr{X}'_n,X'_n) \to (\mathscr{X}'_{n-1},X'_{n-1}) \to \cdots \to (\mathscr{X}'_2, X'_2) \to \overline{\eta} \] is a compactification of this sequence of parameterizing morphisms. 
\begin{proposition} \label{prop:polycompactification}
Let $X$ be a hyperbolic polycurve of dimension $n \geq 1$ over a field $K$ of characteristic zero, and assume that there exists a sequence of parameterizing morphisms  \begin{equation} X = X_n \to X_{n-1} \to \cdots \to X_1 \to X_0 = \Spec K \end{equation} such that for all $i$, the genus of $X_i \rightarrow X_{i-1}$ is greater than or equal to two. Then for any compactification \begin{equation} (\mathscr{X},X) = (\mathscr{X}_n,X_n) \to (\mathscr{X}_{n-1},X_{n-1}) \to \cdots \to (\mathscr{X}_1, X_1) \to \Spec K, \end{equation} of the given sequence of parameterizing morphisms, it holds that $\dim_K H^0(\mathscr{X},\omega_{\mathscr{X}/K})>0$, i.e., $\mathscr{X}$ has a nonzero canonical form. 
\end{proposition}
\begin{proof}
We prove our proposition by induction on $n = \dim X$. When $n=1$ this is obvious, as the genus, by definition, equals $\dim_K H^0(\mathscr{X},\omega_{\mathscr{X}/K})$. Assume the statement holds for all dimensions less than $n$. We write $f \colon \mathscr{X} \rightarrow \mathscr{X}_1$ for the map to the base curve. Since $f$ is a map between smooth varieties over $K$, $f$ is by \stacks{0E9K} an lci-morphism. Thus, by \cite[Theorem 4.9]{LiuAlgebraic}, we have the equality $\omega_{\mathscr{X}/K} = f^* \omega_{\mathscr{X}_1/K} \otimes \omega_{\mathscr{X}/\mathscr{X}_1}$ of dualizing sheaves. Using the projection formula, we obtain that  \[ H^0(\mathscr{X},\omega_{\mathscr{X}/K}) = H^0(\mathscr{X}_1,f_*(f^* \omega_{\mathscr{X}_1/K} \otimes \omega_{\mathscr{X}/\mathscr{X}_1})) \] equals 
  \[ H^0(\mathscr{X}_1,\omega_{\mathscr{X}_1/K} \otimes f_*\omega_{\mathscr{X}/\mathscr{X}_1}). \]  We further have the inequality
  \[ h^0(\mathscr{X}, \omega_{\mathscr{X}/K}) \geq \chi(\mathscr{X}_1,\omega_{\mathscr{X}_1/K} \otimes f_*\omega_{\mathscr{X}/\mathscr{X}_1}), \] and applying Riemann--Roch to the smooth projective curve $\mathscr{X}_1$, the right-hand side equals \[ \deg(\omega_{\mathscr{X}_1/K} \otimes f_*\omega_{\mathscr{X}/\mathscr{X}_1}) + \rk(f_*\omega_{\mathscr{X}/\mathscr{X}_1})(1-g_{\mathscr{X}_1}),\]  where $g_{\mathscr{X}_1}$ is the genus of $\mathscr{X}_1$. Expanding the degree, the above equals \[ \deg(f_*\omega_{\mathscr{X}/\mathscr{X}_1}) + \rk(f_*\omega_{\mathscr{X}/\mathscr{X}_1})(g_{\mathscr{X}_1}-1). \] Note that the rank of the vector bundle $f_*\omega_{\mathscr{X}/\mathscr{X}_1}$ is nonzero by induction: it is a vector bundle, and at the pullback to a geometric generic point $\overline{\eta}$ it equals $H^0(\mathscr{X}_{\overline{\eta}},\omega_{\mathscr{X}_{\overline{\eta}}/\overline{\eta}})$, which by our induction assumption and assumptions on $X$ is nonzero. Further, by construction, $g_{\mathscr{X}_1}-1 >0$, and by \cite[Theorem 0.6]{FujitaKähler}, $f_* \omega_{\mathscr{X}/\mathscr{X}_1}$ has non-negative degree. Thus, $\chi(\mathscr{X}_1,\omega_{\mathscr{X}_1/K} \otimes f_*\omega_{\mathscr{X}/\mathscr{X}_1}) >0$, and our proposition follows.
\end{proof}
\begin{remark} \label{rmk:logbirational}
Let $X$ be a smooth variety over a field $K$ of characteristic zero and let $\mathscr{X}$ be an snc-compactification of $X$ and $D = \mathscr{X} \setminus X$. Then $H^0(\mathscr{X},\omega_{\mathscr{X}/K}(D))$ is independent of the choice of $(\mathscr{X},D)$ in the following sense \cite[\S 11.1]{IitakaAlgebraic}:  any two snc-compactifications $\mathscr{X}_1$ and $\mathscr{X}_2$ are dominated by a third, $\mathscr{X}_3$, and this gives rise to isomorphisms \[ H^0(\mathscr{X}_i,\omega_{\mathscr{X}_i/K}(D_i)) \xrightarrow{\sim} H^0(\mathscr{X}_3,\omega_{\mathscr{X}_3/K}(D_3)) \]  for $i= 1,2$. \end{remark}
Proposition \ref{prop:polycompactification} and Remark \ref{rmk:logbirational} imply that if $X$ admits a sequence of parameterizing morphisms in which every successive curve has genus at least two, then $X$ has a nonzero log-canonical form. Recall from Remark \ref{rmk:etalecoverpoly} that any connected finite étale cover of a hyperbolic polycurve is a hyperbolic polycurve. Further, an arbitrary hyperbolic polycurve $X$ has such a sequence of parameterizing morphisms after passing to a finite étale cover. Moreover, having such a sequence of parameterizing morphisms is preserved under finite étale covers. We obtain:
\begin{corollary}\label{prop:polyhyperbolic}
Let $X$ be a hyperbolic polycurve over a field $K$ of characteristic zero. Then there exists a finite étale cover $Y$ of $X$ such that $Y$ and all finite étale covers of $Y$ have a nonzero log-canonical form. 
\end{corollary}
  \begin{remark}
  It is not true that any hyperbolic polycurve has a nonzero log-canonical form. Let $K$ be an algebraically closed field of characteristic zero, and $n \geq 5$ an odd integer, and consider the $\mu_n$-cover \[ f \colon C \rightarrow \mathbb{P}^1_K \]  defined by the equation $y^n=x(x-1)$. This cover is (totally) ramified exactly at $0,1$ and $\infty$. By Riemann--Hurwitz, $C$ has genus $\dfrac{n-1}{2} \geq 2$. Let $U = \mathbb{P}^1_K - \{0,1,\infty\}$ and consider the projective hyperbolic curve \[ X = f^{-1}(U) \times_K C \rightarrow f^{-1}(U) \]  over $f^{-1}(U)$. By taking the quotient with respect to the diagonal $\mu_n$-action, we obtain a hyperbolic curve $Y = X / \mu_n \rightarrow U$ over $U$, clearly $Y$ is a hyperbolic polycurve over $K$. We denote by $\pi \colon X \rightarrow Y$ the quotient map. Then it is clear that the global sections of the log-canonical bundle $H^0(Y,\omega_{Y/K}^{\log})$ inject via $\pi^*$ (see \cite[Proposition 11.3]{IitakaAlgebraic} for the definition of this pullback) into the $\mu_n$-invariant sections of $H^0(X,\omega_{X/K}^{\log})$. However, $H^0(X,\omega_{X/K}^{\log})$ identifies via the Künneth formula with \[ H^0(C,\omega_{C/K}(D)) \otimes H^0(C,\omega_{C/K}), \] where $D$ is the reduced divisor of $C$ with support at the points above $0,1$ and $\infty$. A weight computation readily shows that no nonzero $\mu_n$-invariant element exists in $H^0(C,\omega_{C/K}(D)) \otimes H^0(C,\omega_{C/K})$, hence $Y$ has no nonzero log-canonical section. \end{remark}
  \begin{remark}
  For proper hyperbolic polycurves, one can obtain a much stronger result. Note that if $X$ is a proper hyperbolic polycurve over $\Spec \CC$, the associated complex analytic space $X(\CC)$ has as its universal cover a bounded domain of holomorphy by Bers' uniformization \cite[Lemma 6.2]{GriffithsComplexZariski}. Since the topological fundamental group of $X(\CC)$ is residually finite, \cite[Corollary C]{YooBergman} implies that there is a finite étale cover $X' \rightarrow X$ such that $X'$ has a canonical bundle which is globally generated and separates tangent vectors. Further, \cite[Corollary C]{YooBergman} shows that the finite étale covers with canonical bundles satisfying this property are cofinal in a certain sense.  By the Lefschetz principle, the same holds for any proper polycurve $X$ over a field of characteristic zero. It is our belief that for any proper polycurve $X$ over a field of characteristic zero, there is a finite étale cover with \emph{very ample} canonical bundle, and that suitable log-generalizations apply to the case when $X$ is not proper.
  \end{remark}

\section{The dominant Hom-conjecture}
\subsection{Cohomologically injective maps}
Recall that if $X$ is a connected qcqs scheme, then for any locally constant constructible sheaf $\mathcal{F}$ on $X$, there is always a functorial comparison map $H^*(\pi_1(X),\mathcal{F}_x) \rightarrow H^*(X,\mathcal{F}),$ where $\mathcal{F}_x$ is the stalk of $\mathcal{F}$ at some geometric point of $X$. 
\begin{construction}
Let $X$ be a scheme of finite type over a field $K$ of characteristic zero and let $\ell$ be a prime. We define the $\ell$-adic cohomology groups $H^*(X,\ZZ_{\ell})$ to be Jannsen's continuous étale cohomology groups \cite{JannsenContinuous}.  If $K$ is algebraically closed, $H^*(X,\ZZ_{\ell})$ agrees with the usual $\ell$-adic cohomology groups.  Further, by \cite[Proposition 5.2]{BhattScholzeProetale}, Jannsen's continuous étale cohomology groups agree with the corresponding pro-étale cohomology groups.

Similarly, for a profinite group $G$, we define $H^*(G,\ZZ_{\ell})$ to be the analogous cohomology groups as in \cite[Section 2]{JannsenContinuous}; they coincide with the cohomology of the complex $C^*(G,\ZZ_{\ell})$ of continuous cochains.  

Finally, we set $H^*(X,\QQ_{\ell}) = H^*(X,\ZZ_{\ell}) \otimes_{\ZZ_{\ell}} \QQ_{\ell}$, and $H^*(\pi_1(X),\QQ_{\ell}) = H^* (\pi_1(X),\ZZ_{\ell}) \otimes_{\ZZ_{\ell}} \QQ_{\ell}$. 
\end{construction}
\begin{definition} \label{def:cohdom}
Let $K$ be a field of characteristic zero, and let $X$ and $Y$ be varieties over $K$.  A continuous map $\varphi \colon \pi_1(X) \rightarrow \pi_1(Y)$ over $G_K$ is \emph{cohomologically injective} if it is open and the induced map \[ \varphi^* \colon H^*(\pi_1(Y_{\overline{K}}),\QQ_{\ell}) \rightarrow H^*(\pi_1(X_{\overline{K}}),\QQ_{\ell}) \] is injective for every prime number $\ell$.  We call $\varphi$ \emph{étale cohomologically injective} if it is open, and the induced map  \[ \varphi^* \colon H^*(\pi_1(Y_{\overline{K}}),\QQ_{\ell}) \rightarrow H^*(\pi_1(X_{\overline{K}}),\QQ_{\ell}) \rightarrow H^*(X_{\overline{K}},\QQ_{\ell}) \] is injective for every prime number $\ell$. We say that $\varphi$ is \emph{stably cohomologically injective} (stably étale cohomologically injective) if for every open subgroup $N \subset \pi_1(Y)$, the induced map $\varphi^{-1}(N) \rightarrow N$ is cohomologically injective (étale cohomologically injective).
\end{definition}
\begin{remark} \label{remark:open}
If one removes the openness condition from the definition of stably étale cohomologically injective, then the remaining condition forces openness of  $\pi_1(X) \rightarrow \pi_1(Y)$ whenever $Y$ is a $K(\pi,1)$ such that the Euler--Poincaré characteristic $\chi(Y,\QQ_{\ell})$ is nonzero\footnote{This was pointed out to us by Jakob Stix.}. Indeed, otherwise we could write $\im \varphi = \bigcap_i \pi_1(Y_i)$, where $Y_i \rightarrow  Y$ are finite étale covers whose degrees tend to $\infty$ as $i$ grows. Note that each $Y_i$ is also a $K(\pi,1)$, and that we have the equality 
\[ \chi( (Y_i)_{\overline{K}},\QQ_{\ell}) = (\pi_1(Y_{\overline{K}}) \colon \pi_1((Y_i)_{\overline{K}})) \cdot \chi(Y_{\overline{K}},\QQ_{\ell}) \] of Euler--Poincaré characteristics. Thus, the absolute value of $\chi(\pi_1((Y_i)_{\overline{K}}),\QQ_{\ell})$ goes to $\infty$ as $i$ grows. On the other hand, $\dim H^i(\pi_1((Y_i)_{\overline{K}}),\QQ_{\ell})$ is bounded by $\dim H^i(X_{\overline{K}},\QQ_{\ell})$, which is finite dimensional, and we obtain a contradiction. Thus $\varphi$ is open; note that the same reasoning applies to show that any stable cohomologically injective map $\pi_1(X) \rightarrow \pi_1(Y)$ is open as soon as $X$ and $Y$ are $K(\pi,1)$s and $\chi(Y,\QQ_{\ell}) \neq 0$.
\end{remark}
\begin{remark} \label{remark:composition}
Note that if $X$ is a $K(\pi,1)$, then if $Y \to X$ is a connected finite étale cover, then also $Y$ is a $K(\pi,1)$. Hence, when $X$ is a $K(\pi,1)$, a map is (stably) étale cohomologically injective if and only if it is (stably) cohomologically injective. Clearly, any composition of (stably) cohomologically injective maps, is (stably) cohomologically injective. Similarly, the class of (stably) étale cohomologically injective maps is closed under composition.
\end{remark}
For varieties $X$ and $Y$ over a field $K$ of characteristic zero, a map $\varphi \colon \Pi(X) \rightarrow \Pi(Y)$ of étale homotopy types over $BG_K$ is defined to be cohomologically injective if the induced map \[ \varphi \colon H^*(Y_{\overline{K}},\QQ_{\ell}) \rightarrow H^*(X_{\overline{K}},\QQ_{\ell}) \] is injective for all primes $\ell$, and the associated map $\varphi_* \colon \pi_1(X) \rightarrow \pi_1(Y)$ is open. It is \emph{stably cohomologically injective} if the map $\pi_1(X) \rightarrow \pi_1(Y)$ is open, and the following additional condition holds. For each $Y' \in Y_{\fet}$, corresponding to the open subgroup $H \subset \pi_1(Y)$, and $X'$ corresponding to $\varphi^{-1}(H) \subset \pi_1(X)$, the induced map $\varphi^* \colon H^*(Y'_{\overline{K}},\QQ_{\ell}) \rightarrow H^*(X'_{\overline{K}},\QQ_{\ell})$, defined from Diagram \ref{eq:diagramcovering}, is injective for every prime number $\ell$.  

For two varieties $X$ and $Y$ over $K$, we denote by \[ \Hom^{\cohdom}_{G_K}(\pi_1(X),\pi_1(Y)) \] the set of $\pi_1(Y_{\overline{K}})$-conjugacy classes of cohomologically injective maps; note that if one representative in the conjugacy class is cohomologically injective, so is any representative. Similarly, we denote \[ \Hom_{BG_K}^{\cohdom}(\Pi(X),\Pi(Y))\] the full sub-anima of $\Hom_{BG_K}(\Pi(X),\Pi(Y))$ spanned by cohomologically injective morphisms. The stable variants are denoted \[
\Hom^{\scohdom}_{G_K}(\pi_1(X),\pi_1(Y))
\quad\text{and}\quad
\Hom_{BG_K}^{\scohdom}(\Pi(X),\Pi(Y))
\] respectively. Lastly, the class of étale cohomologically injective, and stably étale cohomologically injective maps are denoted by $\Hom^{\ecohdom}_{G_K}(\pi_1(X),\pi_1(Y))$ and $\Hom^{\secohdom}_{G_K}(\pi_1(X),\pi_1(Y))$.
\begin{lemma}  \label{lem:kpi1cohdom}
Let $X$ and $Y$ be varieties over a field $K$ of characteristic zero, and assume that $Y$ is a $K(\pi,1)$. Then the isomorphism from Lemma \ref{lemma:mappingspaces} induces bijections
\[ \pi_0(\Hom^{\cohdom}_{BG_K}(\Pi(X),\Pi(Y))) \rightarrow \Hom^{\ecohdom}_{G_K}(\pi_1(X),\pi_1(Y)), \] and  
\[ \pi_0(\Hom^{\scohdom}_{BG_K}(\Pi(X),\Pi(Y))) \rightarrow \Hom^{\secohdom}_{G_K}(\pi_1(X),\pi_1(Y)), \]  respectively.
If $X$ is further a $K(\pi,1)$, the analogous bijections hold if one replaces $\Hom_{G_K}^{\ecohdom}$ and $\Hom_{G_K}^{\secohdom}$ by $\Hom_{G_K}^{\cohdom}$ and $\Hom^{\scohdom}_{G_K}$ respectively.
\end{lemma}
\begin{proof}
The first isomorphism follows immediately, since $\Pi(X) \rightarrow \Pi(Y) = B \pi_1(Y)$ over $BG_K$ factors over a morphism $B \pi_1(X) \rightarrow B \pi_1(Y)$ over $BG_K$. By applying $\pi_1$, we obtain a $\pi_1(Y_{\overline{K}})$-class of morphisms $\pi_1(X) \rightarrow \pi_1(Y)$ over $G_K$, and clearly if $\Pi(X) \rightarrow \Pi(Y)$ is cohomologically injective, $\pi_1(X) \rightarrow \pi_1(Y)$ is étale cohomologically injective by definition. The stable version follows by a similar argument, when we note that if $Y' \rightarrow Y$ is finite étale, then $Y'$ is also a $K(\pi,1)$, c.f. also Diagram \ref{eq:diagramcovering}.  The last statement follows immediately from this, since if $X$ is a $K(\pi,1)$,we have that $\Hom_{G_K}^{\cohdom} (\pi_1(X),\pi_1(Y)) = \Hom^{\ecohdom}_{G_K}(\pi_1(X),\pi_1(Y))$ and $\Hom^{\scohdom}_{G_K}(\pi_1(X),\pi_1(Y)) = \Hom^{\secohdom}_{G_K}(\pi_1(X),\pi_1(Y))$. 
\end{proof} 
\begin{lemma} \label{lemma:dominantcoh}
Let $X$ and $Y$ be proper smooth varieties over $K$. Then any dominant map $f \colon X \rightarrow Y$ over $K$ induces a stably cohomologically injective map \[ f_* \colon \Pi(X) \rightarrow \Pi(Y). \] If $Y$ is moreover a $K(\pi,1)$, then the map $f_* \colon \pi_1(X) \rightarrow \pi_1(Y)$ is stably étale cohomologically injective (a fortiori stably cohomologically injective).
\end{lemma}	
\begin{proof}
It follows from \cite[Prop 1.2.4]{KleimanAlgebraic} that if $X$ and $Y$ are proper smooth varieties over $K$, then any dominant map $f \colon X \rightarrow Y$ induces an injective map $$f^* \colon H^*(Y_{\overline{K}},\QQ_{\ell}) \rightarrow H^*(X_{\overline{K}},\QQ_{\ell})$$ for all primes $\ell$. Since the induced map on $\pi_1$ is open, and the property of being dominant is stable under base-change, the statement for étale homotopy types follows.

The result when $Y$ is a $K(\pi,1)$ follows directly from the above and Lemma \ref{lem:kpi1cohdom}.
\end{proof}
Lemma \ref{lemma:dominantcoh} implies that if $X$ and $Y$ are proper and smooth varieties over $K$, then the natural map 
\[ \Hom^{\dom}_K(X,Y) \rightarrow \pi_0(\Map_{BG_K}(\Pi(X),\Pi(Y))) \] factors through $\pi_0(\Map^{\scohdom}_{BG_K}(\Pi(X),\Pi(Y)))$ (a fortiori, it factors through the cohomologically injective morphisms). If $Y$ is a $K(\pi,1)$, we have a map \[ \Hom^{\dom}_K(X,Y) \rightarrow \Hom_{G_K}^{\secohdom}(\pi_1(X),\pi_1(Y)).\]  In that case, a fortiori, we have a natural map \[ \Hom^{\dom}_K(X,Y) \rightarrow \Hom_{G_K}^{\scohdom}(\pi_1(X),\pi_1(Y)).\] 

Now we wish to make an analogous definition for not necessarily proper varieties. The correct notion should involve a ``group-theoretical'' version of compactly supported cohomology. However, as we cannot currently perform such a construction, we give the following ad-hoc definition over $p$-adic fields $K$. Given a smooth variety $X$ over a $p$-adic field $K$, pick an snc-compactification $X \subset \overline{X}$ and let $D$ be the divisor corresponding to $\overline{X} \setminus X$. Then it is well-known that $(H^n(X_{\overline{K}},\QQ_p) \otimes \CC_p(n))^{G_K}$ identifies with $H^0(\overline{X},\Omega^n_{\overline{X}/K}(D))$ (see \cite[Remark 1.2.6]{BhattHodgeTate}). We note that by the same argument as in Remark \ref{rmk:logbirational}, this latter group is independent of the given snc-compactification. We write $\HT(X) = \oplus_n H^0(\overline{X},\Omega^n_{\overline{X}/K}(D))$, and note that if $Y$ is a $K(\pi,1)$, any map $\varphi \colon \pi_1(X) \rightarrow \pi_1(Y)$ over $G_K$ induces a map $\HT(Y) \rightarrow \HT(X)$, by the Hodge--Tate decomposition.
\begin{definition} \label{def:hodgetateinjective}
Let $K$ be a $p$-adic field, and let $X$ and $Y$ be smooth varieties over $K$, and assume that $Y$ is a $K(\pi,1)$. A map $\varphi \colon \pi_1(X) \rightarrow \pi_1(Y)$ over $G_K$ is \emph{Hodge--Tate injective} if it is open, and the induced map $\HT(Y) \rightarrow \HT(X)$ is injective. We say that $\varphi$ is \emph{stably Hodge--Tate injective} if for each open subgroup $N \subset \pi_1(Y)$, the induced map $\varphi^{-1}(N) \rightarrow N$ is Hodge--Tate injective.
\end{definition}
\begin{remark}
The above definition is relatively group-theoretical relative to $G_K$, but not \emph{absolutely} group-theoretical, in sense that it involves $\CC_p$, which is not determined by $G_K$ as a profinite group.  One can however reconstruct $\CC_p$ from $G_K$ together with its ramification filtration (see \cite[Proposition 2.2]{MochizukiGrothendieckLocal}).
\end{remark}
A morphism $\Pi(X) \rightarrow \Pi(Y)$ of étale homotopy types over $BG_K$, with $X$ and $Y$ smooth varieties, is Hodge--Tate injective if it is open on $\pi_1$ and the analogous cohomological condition of Definition \ref{def:hodgetateinjective} is fulfilled; a stable version is defined analogously. If $X$ and $Y$ are smooth varieties over a $p$-adic field $K$, and $Y$ is in addition a $K(\pi,1)$, we claim that taking $\pi_1$ induces a map \[\Hom^{\dom}_K(X,Y) \rightarrow \Hom_{G_K}^{\HTinj}(\pi_1(X),\pi_1(Y)), \]  where the right-hand side consists of conjugacy classes of Hodge--Tate injective maps; a similar statement holds for étale homotopy types and stably Hodge--Tate injective maps. Indeed: given a dominant map $f \colon X \rightarrow Y$, extend it to a dominant map $\overline{f} \colon \overline{X} \rightarrow \overline{Y}$ between some smooth snc-compactifications. For any $n$, the induced map $(H^n(Y_{\overline{K}},\QQ_p) \otimes \CC_p(n))^{G_K} \rightarrow (H^n(X_{\overline{K}},\QQ_p) \otimes \CC_p(n))^{G_K}$ on Hodge--Tate weight $n$-subspaces corresponds to the pullback of (global sections) of the respective log-canonical bundles (see \cite[Remark 1.2.6]{BhattHodgeTate}) by $\overline{f}$, which is injective. This shows that any dominant map induces a (stably) Hodge--Tate injective map.

Note that over a $p$-adic local field, every (stably) étale cohomologically injective map is (stably) Hodge--Tate injective. 
\subsection{Proof of the dominant Hom-conjecture}
In this subsection we prove the (relative) dominant Hom-conjecture for proper hyperbolic polycurves over sub-$p$-adic fields. We also establish a variant over $p$-adic fields involving Hodge--Tate injectivity. As a consequence, we derive the (relative) Isom-conjecture for hyperbolic polycurves over sub-$p$-adic fields.
\begin{theorem} \label{thm:dominant}
  Let $K$ be a sub-$p$-adic field.
  \begin{enumerate}
  
  \item
  If $X$ is a smooth proper variety which is a $K(\pi,1)$, and $Y$ is a proper hyperbolic polycurve over $K$, then the natural map
  \[
  \Hom^{\dom}_K(X,Y)
  \longrightarrow
  \ \Hom^{\scohdom}_{G_K}(\pi_1(X),\pi_1(Y))
  \]
  is bijective. If $Y$ is in addition of dimension at most two, then the above bijectivity holds just under the assumption that $X$ is a smooth proper variety.
  \item
  Let $X$ be a smooth proper variety over $K$ and $Y$ a proper hyperbolic polycurve over $K$. Then the natural map
  \[
  \Hom^{\dom}_K(X,Y)
  \longrightarrow
  \ \pi_0(\Map^{\scohdom}_{BG_K}(\Pi(X),\Pi(Y)))
  \]
  is bijective. 
  \item
  If $K$ is a $p$-adic field, the above statements remain valid without assuming $X$ and $Y$ are proper, provided one replaces cohomologically injective maps with Hodge--Tate injective maps.
  \end{enumerate}
  \begin{remark}
One can obtain a more general version of Theorem \ref{thm:dominant} (iii). More precisely, if $K$ is any sub-$p$-adic field, one could define a map $\varphi \colon \pi_1(X) \rightarrow \pi_1(Y)$ between fundamental groups of smooth varieties over $G_K$ to be Hodge--Tate injective if it is so after base-change and specialization to a $p$-adic field (see the specialization argument in Proposition \ref{prop:hodgepoly} for details). Then, Theorem \ref{thm:dominant} (iii) can be strengthened, in that it allows one to say the same for an arbitrary sub-$p$-adic field $K$. However, since the definition of \emph{Hodge--Tate injective} is involved when $K$ is not a $p$-adic local field, we choose not to expand on this further and leave details to the interested reader.
  \end{remark}
\end{theorem}
\begin{theorem} \label{cor:isom}
Let $X$ and $Y$ be hyperbolic polycurves over a sub-$p$-adic field $K$. Then the natural map $\Isom_K(X,Y) \rightarrow \Isom_{G_K}(\pi_1(X),\pi_1(Y))$ is bijective.
\end{theorem} 
\begin{remark}
As in \cite[Theorem 6.1]{SchmidtStix}, the results of Pop \cite{PopGrothendieck,PopAlterations} on birational anabelian geometry allow one to derive an \emph{absolute} version of the Isom-conjecture in certain situations. More precisely, let $X$ and $Y$ be hyperbolic polycurves defined over possibly different fields which are finitely generated over $\QQ$.  Then the natural map $\Isom(X,Y) \rightarrow \Isom(\pi_1(X),\pi_1(Y))$ is bijective. See the proof of \cite[Theorem 6.1]{SchmidtStix} for details.
\end{remark}
Grothendieck conjectured that for a hyperbolic polycurve $Y$, and any smooth variety $X$, the natural map \[ \Hom^{\dom}_K(X,Y) \rightarrow \Hom^{\open}_{G_K}(\pi_1(X),\pi_1(Y)) \] is bijective. The following example shows that this fails in $\dim Y > 1$. It further shows that, as soon as $\dim Y >2$, any subcategory of the category of smooth projective varieties for which the image of $\Hom^{\dom}_K(X,Y)$ inside $\Hom_{G_K}^{\open}(\pi_1(X),\pi_1(Y))$ can be characterized group-theoretically must be highly restricted.
\begin{example}\label{ex:grothendieckcounter}
Let $K$ be a field of characteristic zero, and let $Y$ be a proper hyperbolic polycurve of dimension greater than one over $K$. For any smooth hyperplane section $i \colon H  \subset Y$ over $K$, Lefschetz hyperplane theorem shows that we have a natural surjection $i_* \colon \pi_1(H) \rightarrow \pi_1(Y)$ over $G_K$. We thus see that $\Hom^{\open}_{G_K}(\pi_1(H),\pi_1(Y))$ is non-empty, containing the map $i_*$, but $\Hom^{\dom}_K(H,Y) = \emptyset$ since $\dim H < \dim Y$.

Further, if $\dim Y \geq 3$, by iteratively taking smooth hyperplane sections,  one sees that one cannot in general characterize the dominant morphisms in a group-theoretical manner inside \[ \Hom_{G_K}^{\open}(\pi_1(X),\pi_1(Y))\]  for $X$ smooth and projective.

In fact, one cannot in general even distinguish dominant morphisms in the category of smooth proper varieties admitting a dominant map onto $Y$. To see this, let $K$ be a sub-$p$-adic field and let $i \colon H \subset Y$ be any smooth hyperplane section over $K$. Since $\dim(Y) \geq 3$, the map $i_* \colon \pi_1(H) \rightarrow \pi_1(Y)$ is an isomorphism. Set $X = Y \times_K H$, then clearly $X$ admits a dominant map onto $Y$. Further, $\pi_1(X) \cong \pi_1(Y) \times_{G_K} \pi_1(H)$. The projections $\pr_1 \colon \pi_1(X) \rightarrow \pi_1(Y)$ and \[ i_* \circ  \pr_2 \colon \pi_1(X) \rightarrow \pi_1(H) \rightarrow \pi_1(Y) \] are isomorphic by the automorphism $\pi_1(X) \rightarrow \pi_1(X)$ taking $(s,t) \in \pi_1(X)$ to $(i_*(t),i^{-1}_*(s))$. However, the map $i_* \circ (\pr_2)_* \colon \pi_1(X) \rightarrow \pi_1(Y)$ is \emph{not} induced by a dominant morphism of schemes.  Indeed, it is the image of the map $X \rightarrow H \rightarrow Y$, and since $\Hom_K(X,Y) \rightarrow \Hom_{G_K}(\pi_1(X),\pi_1(Y))$ is injective \cite[Proposition 3.2]{HoshiGrothendieck}, our claim follows. Thus, even though $\pr_1$ is induced by a dominant morphism, the isomorphic map $i \circ \pr_2$ is not.  \end{example}

To prove our main theorems, we recall the following definition, due to Hoshi \cite[Def.~3.6]{HoshiGrothendieck}.
\begin{definition} \label{def:polynondeg}
Let $X$ and $Y$ be normal varieties over a field $K$ of characteristic zero. 
\begin{enumerate}  
  \item An open map $\varphi \colon \pi_1(X) \rightarrow \pi_1(Y)$ over $G_K$ is nondegenerate if the following condition holds: for any non-empty Zariski open $U \subset X$ and smooth, surjective, geometrically connected map $g \colon U \rightarrow Z$ over $K$ with $Z$ normal and $\dim(Z) < \dim(Y)$, the induced map $\pi_1(U) \rightarrow \pi_1(Y)$ does not factor through $\pi_1(U) \rightarrow \pi_1(Z)$. 
  \item If $Y$ is a polycurve of dimension $n$, then an open map $\varphi \colon \pi_1(X) \rightarrow \pi_1(Y)$ is poly-nondegenerate if there is a sequence of parameterizing maps \[ Y = Y_n \rightarrow Y_{n-1} \rightarrow \cdots \rightarrow Y_1 \rightarrow Y_0 = \Spec K \] such that for each $0 \leq i \leq n$, the composite $\pi_1(X) \rightarrow \pi_1(Y_i)$ is nondegenerate. 
\end{enumerate}
\end{definition}

\begin{proposition} \label{prop:hodgepoly}
Let $K$ be a sub-$p$-adic field, and $X$ and $Y$ smooth varieties over $K$, and assume that $Y$ is a $K(\pi,1)$. Assume that there is a connected finite étale cover $Y' \rightarrow Y$ such that $H^0(\overline{Y'},\omega_{\overline{Y'}}(D')) \neq 0$ for some snc-compactification $Y' \subset \overline{Y'}$ with boundary divisor $D'$. 
\begin{enumerate}
\item If $K$ is a $p$-adic local field, every stably Hodge--Tate injective map \[ \varphi \colon \pi_1(X) \rightarrow \pi_1(Y) \] is nondegenerate.
\item Assume that $\pi_1(Y_{\overline{K}})$ is slim. Then every stably étale cohomologically injective map \[ \varphi \colon \pi_1(X) \rightarrow \pi_1(Y) \] is nondegenerate. 
\end{enumerate} 
If $Y$ admits a nonzero log-canonical form, the above conclusions hold for étale cohomologically injective and Hodge--Tate injective maps, respectively.
\end{proposition}
\begin{proof}
We first prove (1). Note that a map is nondegenerate if it is so after passing to a finite étale cover of $Y$. Indeed, given $\varphi \colon \pi_1(X) \rightarrow \pi_1(Y)$, and an étale cover  $Y' \rightarrow Y$ we obtain a map $\pi_1(X') \rightarrow \pi_1(Y')$ where $X' \rightarrow X$ is the connected finite étale cover corresponding to $\varphi^{-1}(H)$. Any factorization as in Definition \ref{def:polynondeg} of $\pi_1(U) \rightarrow \pi_1(X) \rightarrow \pi_1(Y)$ for some Zariski open $U \subset X$, would give a factorization of $\pi_1(U') \rightarrow \pi_1(X') \rightarrow \pi_1(Y')$, where $U' = U \times_X X'$. Thus, we may assume that $H^0(\overline{Y},\Omega^n_{\overline{Y}/K}(D)) \neq 0$ for some snc-compactification $Y \subset \overline{Y}$ with boundary $D$. 
Let $\varphi \colon \pi_1(X) \rightarrow \pi_1(Y)$  be a Hodge--Tate injective map. Suppose, for contradiction, that there exists a Zariski open $U \subset X$, a normal scheme $Z$ with $\dim Z < \dim Y = n$, and a smooth, surjective, geometrically connected map $g \colon U \rightarrow Z$ such that the induced map $\varphi \colon \pi_1(U) \rightarrow \pi_1(Y)$ factors through $g_* \colon \pi_1(U) \rightarrow \pi_1(Z)$. Since $Y$ is a $K(\pi,1)$, we obtain an induced map \[ H^*(Y_{\overline{K}},\QQ_p) \rightarrow H^*(U_{\overline{K}},\QQ_p), \] which factors over $H^*(Y_{\overline{K}},\QQ_p) \rightarrow H^*(Z_{\overline{K}},\QQ_p)$ as $G_K$-modules by the factorization on fundamental groups.

Choose compatible snc-compactifications $\overline{X}, \overline{Y}, \overline{Z},\overline{U},$ of $X,Y,Z$ and $U$, and write $D_X,D_Y,D_Z$ and $D_U$ for the corresponding snc divisors. By functoriality of the Hodge--Tate decomposition, the map \[ H^n(Y_{\overline{K}},\QQ_p) \rightarrow H^n(X_{\overline{K}},\QQ_p) \rightarrow H^n(U_{\overline{K}},\QQ_p) \] induces a map \[ H^0(\overline{Y},\Omega^n_{\overline{Y}/K}(D_Y)) \rightarrow H^0(\overline{X},\Omega^n_{\overline{X}/K}(D_X)) \rightarrow  H^0(\overline{U},\Omega^n_{\overline{U}/K}(D_U)). \]  However, by functoriality, this composite factors through $(H^n(Z_{\overline{K}},\QQ_p) \otimes \CC_p(n))^{G_K}$. This subspace is zero, since $\dim(Z) < n$, and $Z$ is smooth. On the other hand, the map $H^0(\overline{Y},\Omega^n_{\overline{Y}/K}(D_Y)) \rightarrow H^0(\overline{U},\Omega^n_{\overline{U}/K}(D_U))$ is injective. Indeed, by assumption, $H^0(\overline{Y},\Omega^n_{\overline{Y}/K}(D_Y)) \rightarrow H^0(\overline{X},\Omega^n_{\overline{X}/K}(D_X))$ is injective. Further, $H^0(\overline{X},\Omega^n_{\overline{X}/K}(D_X)) \rightarrow  H^0(\overline{U},\Omega^n_{\overline{U}/K}(D_U))$ is injective: the map $\overline{U} \rightarrow \overline{X}$ is birational. Since any stably étale cohomologically injective map over a $p$-adic local field is clearly stably Hodge--Tate injective, (2) is proven when $K$ is a $p$-adic field. The above argument also shows that if $Y$ has a nonzero log-canonical form, then we only need to assume that $\varphi$ is étale cohomologically injective.

We now prove (2) for a general sub-$p$-adic field $K$. We claim that if the proposition holds over a field $K$, it is known for all subfields of $K$. Indeed, we can first as in the proof of (1) assume that $H^0(\overline{Y},\Omega^n_{\overline{Y}/K}(D)) \neq 0$. Note that the property of being a $K(\pi,1)$, as well as the existence of a nonzero log-canonical form, is stable under base-change along field extensions in characteristic zero. Thus, given a factorization as in Definition \ref{def:polynondeg}, since also fundamental groups commute with base-change along field extensions in characteristic zero, and the base-change of the map $\pi_1(X) \rightarrow \pi_1(Y)$ is still étale cohomologically injective, we obtain a contradiction. Thus, the statement is known for all fields which are subfields of a $p$-adic local field. A general sub-$p$-adic field $K$ is by definition a subfield of a finitely generated extension of $\QQ_p$ of transcendence degree $d \geq 0$. To prove our statement, we can assume that $K$ is a field which is finitely generated over $\QQ_p$, since the statement then holds for all subfields of $K$. We can find a subfield $L \subset K$ such that $L$ is a $p$-adic local field, and $K$ is the function field of a variety $V$ over $L$. We can once again assume that $H^0(\overline{Y},\Omega^n_{\overline{Y}/K}(D)) \neq 0$ for some snc-compactification $Y \subset \overline{Y}$ with boundary $D$. Let now $\varphi$ be étale cohomologically injective and suppose that it is degenerate. Then there is a non-empty Zariski open $U \subset X$, and a smooth, surjective, geometrically connected map $U \rightarrow Z$ with $\dim(Z) < \dim(Y)$ such that the induced map $\pi_1(U) \rightarrow \pi_1(Y)$ factors through $\pi_1(U) \rightarrow \pi_1(Z).$  Choose snc-compactifications $\overline{U},\overline{X},\overline{Y},\overline{Z}$ of $U,X,Y,Z$ respectively. After possibly shrinking $V$, we can arrange that there exist smooth models $\mathcal{U},\mathcal{X},\mathcal{Y},\mathcal{Z}$ over $V$ of $U,X,Y,Z$, and relative snc-compactifications  $\overline{\mathcal{U}},\overline{\mathcal{X}},\overline{\mathcal{Y}},\overline{\mathcal{Z}}$ respectively, which are smooth over $V$; denote the corresponding relative divisors $\mathcal{D}_U,\mathcal{D}_X,\mathcal{D}_Y,\mathcal{D}_Z$. We can further arrange so that there is an open immersion $\mathcal{U} \rightarrow \mathcal{X}$ over $V$, which on the generic fiber recovers the open immersion $U \rightarrow X$. Similarly, we can assume that there is a map $\mathcal{U} \rightarrow \mathcal{Z}$ over $V$ which is smooth, surjective and geometrically connected, and on the generic fiber is isomorphic to $U \rightarrow Z$. By possibly shrinking further, we may assume that $H^0(\overline{\mathcal{Y}},\Omega^n_{\overline{\mathcal{Y}}/V}(\mathcal{D}_Y)) \neq 0$. Let $v \in V$ be a closed point with the property that $H^0(\overline{\mathcal{Y}}_v,\Omega^n_{\overline{\mathcal{Y}}_v/\Spec k(v)}(\mathcal{D}_v)) \neq 0$, where $\overline{\mathcal{Y}}_v$ is the fiber at $v$. There is a henselian discrete valuation ring $\Spec W$, with fraction field $\eta = \Spec K(W)$ and closed point $w$, and a dominant map $W \rightarrow V$ over $L$ which has the property that the closed point of $W$ goes to $v$, and induces an isomorphism on residue fields. We now pull-back the given factorization along $G_{K(W)} \rightarrow G_K$, and all geometric objects along the map $W \rightarrow V$. Let $\Spec \overline{K(W)} = \overline{\eta} \rightarrow W$ be a geometric point lying over the generic point, and $\overline{s} \rightarrow W$ be a geometric point lying over the closed point, chosen so that $\overline{s}$ is an étale specialization of $\overline{\eta}$. Note that the geometric objects pulled back to $W$ are smooth, and admit relative snc-compactifications by construction. Since the residue field is of characteristic zero, the specialization map $\pi_1(Y_{\overline{\eta}}) \rightarrow \pi_1(Y_{\overline{s}})$ is an isomorphism. Note further that the outer action of $G_{K(W)}$ on the fundamental groups of all objects appearing in the factorization of $\pi_1(U_{K(W)}) \rightarrow \pi_1(Y_{K(W)})$ factors through $G_s = \Gal(\overline{k(s)}/k(s))$. We now apply Proposition \ref{prop:homouter}, to specialize the given factorization of  $\pi_1(U_{K(W)}) \rightarrow \pi_1(Y_{K(W)})$ through $\pi_1(U_{K(W)}) \rightarrow \pi_1(Z_{K(W)})$ to an analogous factorization of $\pi_1(U_{k(s)}) \rightarrow \pi_1(Y_{k(s)})$ over $G_{k(s)}$.  Lemma \ref{lemma:kpi1family},  and the subsequent discussion implies that $Y_s$ is a $K(\pi,1)$. Further, Diagram \ref{eq:diagram} and the fact that $H^*(X_{\overline{s}},\QQ_{\ell}) =( \lim_n H^*(X_{\overline{s}},\ZZ/\ell^n) ) \otimes \QQ_{\ell}$, implies that $\pi_1(X_s) \rightarrow \pi_1(Y_s)$ is still étale cohomologically injective, and the map is clearly open. However, since $k(s)$ is a $p$-adic local field, we obtain a contradiction to what we proved in (1).
\end{proof}
\begin{corollary} \label{cor:polynondeg}
Let $K$ be a sub-$p$-adic field, and $X$ a smooth variety over $K$, and $Y$ a hyperbolic polycurve over $K$. Then every stably étale cohomologically injective map $\varphi \colon \pi_1(X) \rightarrow \pi_1(Y)$ is polynondegenerate. Further, if $Y$ is a hyperbolic polycurve over $K$ with the property that there exists a sequence of parameterizing morphisms   \[ Y = Y_n \rightarrow Y_{n-1} \rightarrow \cdots \rightarrow Y_1 \rightarrow Y_0 = \Spec K \]  with the property that each $Y_m$, $m \leq n$ has a nonzero log-canonical form, then any étale cohomologically injective map is polynondegenerate.
\end{corollary}
\begin{proof}
We first prove the statement over a $p$-adic local field. Suppose given a stably étale cohomologically injective $\varphi \colon \pi_1(X) \rightarrow \pi_1(Y)$. To prove our statement, we can, by Corollary \ref{prop:polyhyperbolic}, assume  that $Y$ is a hyperbolic polycurve with the property that there is a sequence of parameterizing morphisms  \begin{equation} Y = Y_n \rightarrow Y_{n-1} \rightarrow \cdots \rightarrow Y_1 \rightarrow Y_0 = \Spec K  \label{eq:polydeg} \end{equation} such that $Y_m$ has a nonzero log-canonical form for every $m \leq n$. Since $Y \rightarrow Y_m$ is dominant for any $m \leq n$, the associated map $\pi_1(Y) \rightarrow \pi_1(Y_m)$ is stably Hodge--Tate injective. Since stably Hodge--Tate injective maps are closed under composition, Proposition \ref{prop:hodgepoly} proves that $\pi_1(X) \rightarrow \pi_1(Y_m)$ is nondegenerate. Thus $\varphi$ is polynondegenerate. 
For the case of a general field, note that if our Corollary is true for one field $K$, then it holds for all subfields $L \subset K$. Indeed, after replacing $Y$ over $K$ by a finite étale cover, we can assume that there exists a sequence of parameterizing morphisms as in \ref{eq:polydeg} such that each $Y_m$ has a nonzero log-canonical form. Then, since fundamental groups, as well as the notion of being étale cohomologically injective commutes with base-change along field extension in characteristic zero, we obtain the proposition for $K$. Suppose now that $\varphi \colon \pi_1(X) \rightarrow \pi_1(Y)$ is an open map which is stably étale cohomologically injective, and suppose there exists an $m \leq n$, a Zariski open $U \subset X$, a smooth, surjective, geometrically connected map $U \rightarrow Z$, with $\dim(Z) < m$, such that the map $\pi_1(U) \rightarrow \pi_1(X) \xrightarrow{\varphi} \pi_1(Y) \rightarrow \pi_1(Y_m)$ factors over $\pi_1(U) \rightarrow \pi_1(Z)$. By replacing $Y$ by a finite étale cover, we can assume that $Y_m$ has a nonzero log-canonical form. As in the proof of Proposition \ref{prop:hodgepoly}, we can reduce to the case when $K$ is the function field of a variety $V$ over $L$, where $L$ is a $p$-adic field. By possibly shrinking $V$, we can find smooth models $\mathcal{U}, \mathcal{X}, \mathcal{Z}$, of $U,X,Z$ respectively, and corresponding relative snc-compactifications  $\overline{\mathcal{U}},\overline{\mathcal{X}},\overline{\mathcal{Y}},\overline{\mathcal{Z}}$, smooth and proper over $V$. Such models and relative snc-compactifications can also be assumed to exist for $Y_i$, for each $1 \leq i \leq n$; write $\mathcal{Y}_i$ for the corresponding smooth model, and $\overline{\mathcal{Y}_i}$ for the relative snc-compactification. We can further shrink so that $\mathcal{Y}$ is a hyperbolic polycurve over $V$, and even assume that $\mathcal{Y}  = \mathcal{Y}_n \rightarrow \mathcal{Y}_{n-1} \rightarrow \cdots \rightarrow \mathcal{Y}_1 \rightarrow V$ is a sequence of parameterizing maps, which coincide with the given one on the generic fiber. We also arrange it so that there is a map $\mathcal{U} \rightarrow \mathcal{X}$ is an open immersion, and a map $\mathcal{U} \rightarrow \mathcal{Z}$ which is smooth, surjective geometrically connected, and so that on the generic fibers, these maps recover $U \rightarrow X$ and $U \rightarrow Z$. We now find a discrete henselian valuation ring, and a dominant map $\Spec W \rightarrow V$ where the closed point of $\Spec W$ goes to a closed point of $V$, and further induces an isomorphism on residue fields. By specializing, arguing as in Proposition \ref{prop:hodgepoly}, we would obtain a factorization of  $\pi_1(U_{k(s)}) \rightarrow \pi_1(Y_{k(s)})$ through $\pi_1(U_{k(s)}) \rightarrow \pi_1(Z_{k(s)})$, which would be a contradiction since $\pi_1(X_{k(s)}) \rightarrow \pi_1(Y_{k(s)})$ is still étale cohomologically injective.
\end{proof}
\begin{remark} \label{rmk:saver}
Assume that $Y$ is a hyperbolic polycurve of dimension $n$. We call a map $\varphi \colon \pi_1(X) \rightarrow \pi_1(Y)$ for stably étale cohomologically injective in degree up to $n$ if the injectivity of Definition \ref{def:cohdom} holds in degree less than or equal to $n$. Thus, we do not require the map $H^m(\pi_1(Y_{\overline{K}}),\QQ_{\ell}) \rightarrow H^m(X_{\overline{K}},\QQ_{\ell})$ to be injective if $m > n$. By inspecting the proof of Corollary \ref{cor:polynondeg} and Proposition \ref{prop:hodgepoly}, any stably étale cohomologically injective map in degree up to $n$, $\varphi \colon \pi_1(X) \rightarrow \pi_1(Y)$, is polynondegenerate.\end{remark}
If $X$ and $Y$ are $K(\pi,1)$s, a map is (stably) étale cohomologically injective map if and only if it is cohomologically injective. We easily derive:
\begin{corollary} \label{cor:hodgepoly} 
Let $K$ be a sub-$p$-adic field, and let $X$ and $Y$ be smooth varieties, and assume that $Y$ is a hyperbolic polycurve, and that $X$ is a $K(\pi,1)$. Then every stably cohomologically injective map $\pi_1(X) \rightarrow \pi_1(Y)$ over $G_K$ is polynondegenerate. Further, if $Y$  is a hyperbolic polycurve over $K$ with the property that there exists a sequence of parameterizing morphisms   \[ Y = Y_n \rightarrow Y_{n-1} \rightarrow \cdots \rightarrow Y_1 \rightarrow Y_0 = \Spec K \]  with the property that each $Y_m$, $m \leq n$ has a nonzero log-canonical form, then any cohomologically injective map is in fact polynondegenerate. If the dimension of $Y$ is at most two, the analogous statements hold if one just assumes that $X$ is smooth.
\end{corollary}
\begin{proof}
The only non-trivial part to prove is the statement that when $Y$ has dimension at most two, any stably cohomologically injective map $\varphi \colon \pi_1(X) \rightarrow \pi_1(Y)$ is polynondegenerate. This will follow from Remark \ref{rmk:saver} when we note that any such map is stably étale cohomologically injective in degree up to two. Let $\tilde{X} \rightarrow X_{\overline{K}}$ be ``the'' universal cover of $X_{\overline{K}}$. We consider the spectral sequence with $E_2$-term $E^{p,q}_2 = H^p(\pi_1(X_{\overline{K}}),(H^q(\tilde{X},\ZZ/\ell^n)_n))$, which converges to $H^{p+q}(X_{\overline{K}},\ZZ_{\ell})$, see \cite[Theorem 3.3]{JannsenContinuous}. The cohomology group $ H^p(\pi_1(X_{\overline{K}}),(H^q(\tilde{X},\ZZ/\ell^n)_n))$ is continuous étale cohomology with coefficients in the inverse system  $(H^q(\tilde{X},\ZZ/\ell^n))_n$. As $\tilde{X}$ is simply connected, $H^1(\tilde{X},\ZZ/\ell^n)=0$ for all $n \geq 1$. This proves that the $E_2$-term $E_2^{0,1}$ vanishes. We then get, in addition to the edge map $H^1(\pi_1(X_{\overline{K}}),\ZZ_{\ell}) \rightarrow H^1(X_{\overline{K}},\ZZ_{\ell})$ being an isomorphism, that the map $H^2(\pi_1(X_{\overline{K}}),\ZZ_{\ell}) \rightarrow H^2(X_{\overline{K}},\ZZ_{\ell})$ is injective. By tensoring with $\QQ_{\ell}$, we see that any stably cohomologically injective map is stably étale cohomologically injective in degree up to two, hence polynondegenerate.
\end{proof}
We now recall the following Theorem which is due to Hoshi \cite[Theorem 3.7]{HoshiGrothendieck}.
\begin{theorem} \label{thm:hoshi}
Let $K$ be a sub-$p$-adic field, $X$ a normal variety, and $Y$ a hyperbolic polycurve. Then the natural map 
$\Hom_K^{\dom}(X,Y) \rightarrow \Hom_{G_K}(\pi_1(X),\pi_1(Y))$ determines a bijection \[ \Hom_K^{\dom}(X,Y) \rightarrow \Hom^{\PND}_{G_K}(\pi_1(X),\pi_1(Y)), \]where the latter is the set of (conjugacy classes of) polynondegenerate maps.
\end{theorem}
We can now prove our main theorem.
\begin{proof}[Proof of Theorem \ref{thm:dominant}]
We first prove (1) when $X$ is a smooth and proper $K(\pi,1)$. By Corollary \ref{cor:hodgepoly}, any stably cohomologically injective map $\varphi \colon \pi_1(X) \rightarrow \pi_1(Y)$ is polynondegenerate. We thus have a factorization \[ \Hom^{\dom}_K(X,Y) \rightarrow \Hom_{G_K}^{\scohdom}(\pi_1(X),\pi_1(Y)) \rightarrow \Hom^{\PND}_{G_K}(\pi_1(X),\pi_1(Y)) \] of the bijection $\Hom^{\dom}_K(X,Y) \rightarrow  \Hom^{\PND}_{G_K}(\pi_1(X),\pi_1(Y))$ from Theorem \ref{thm:hoshi}. Since \[ \Hom^{\scohdom}_{G_K}(\pi_1(X),\pi_1(Y)) \rightarrow \Hom^{\PND}_{G_K}(\pi_1(X),\pi_1(Y)) \] is injective, $\Hom^{\dom}_K(X,Y) \rightarrow \Hom^{\scohdom}_{G_K}(\pi_1(X),\pi_1(Y))$ is a bijection. Thus (1) follows when $X$ is a smooth and proper $K(\pi,1)$. The result when $Y$ is of dimension at most two, and $X$ is an arbitrary smooth and proper variety follows along exactly the same lines using the last statement of Corollary \ref{cor:hodgepoly}. For (2), we note that by Lemma \ref{lem:kpi1cohdom}, the set $\pi_0(\Map^{\scohdom}_{BG_K}(\Pi(X),\Pi(Y)))$ equals $\Hom^{\secohdom}_{G_K}(\pi_1(X),\pi_1(Y))$. The same argument as in (1) now proves that the map $\Hom^{\dom}_K(X,Y) \rightarrow \pi_0(\Map^{\scohdom}_{BG_K}(\Pi(X),\Pi(Y)))$ is a bijection. The proof of (3) is analogous to the proofs of (1) and (2).
\end{proof}
\begin{proof}[Proof of Theorem \ref{cor:isom}]
There is a natural injective map \[ \Isom_K(X,Y) \rightarrow \Isom_{G_K}(\pi_1(X),\pi_1(Y)) \subset \Hom^{\scohdom}_{G_K}(\pi_1(X),\pi_1(Y)). \] By Corollary \ref{cor:polynondeg} and Theorem \ref{thm:hoshi}, we obtain from any isomorphism $\varphi \colon \pi_1(X) \rightarrow \pi_1(Y)$ a dominant map $f \colon X \rightarrow Y$ of polycurves with the property that \[ f_* = \varphi \colon \pi_1(X) \rightarrow \pi_1(Y) \] is an isomorphism. By applying the same reasoning to $\varphi^{-1} \colon \pi_1(Y) \rightarrow \pi_1(X)$, we obtain a dominant map $g \colon Y \rightarrow X$ of polycurves such that $g_*= \varphi^{-1} \colon \pi_1(Y) \rightarrow \pi_1(X)$. Then we see that $f \circ g$ and $g \circ f$ are the identity since the maps $\Hom_K(X,X) \rightarrow \Hom_{G_K}(\pi_1(X),\pi_1(X))$ and $\Hom_K(Y,Y) \rightarrow \Hom_{G_K}(\pi_1(Y),\pi_1(Y))$ are injective \cite[Proposition 3.2]{HoshiGrothendieck}. 
\end{proof}

\DeclareFieldFormat{labelnumberwidth}{{#1\adddot\midsentence}}
\printbibliography
\end{document}